\documentclass{amsart}
\usepackage[utf8]{inputenc}
\usepackage{amsmath}
\usepackage{amssymb}
\usepackage{caption}
\usepackage{amsthm}
\usepackage[hidelinks]{hyperref}
\usepackage{cleveref}
\usepackage[justification=centering]{caption}
\crefname{lemma}{Lemma}{Lemmas}
\crefname{theorem}{Theorem}{Theorems}
\usepackage{xcolor}
\usepackage{comment}
\usepackage{graphicx}
\usepackage{tikz}
\usepackage{tikz-cd}
\usepackage{mathrsfs}
\usepackage{enumitem}
\usepackage[T1]{fontenc}
\usetikzlibrary{shapes.geometric}
\makeatletter
\def\@settitle{\begin{center}%
  \baselineskip14\p@\relax
  \bfseries
  \uppercasenonmath\@title
  \@title
  \ifx\@subtitle\@empty\else
     \\[1ex]\uppercasenonmath\@subtitle
     \footnotesize\mdseries\@subtitle
  \fi
  \end{center}%
}
\def\subtitle#1{\gdef\@subtitle{#1}}
\def\@subtitle{}
\makeatother

\newtheorem{theorem}{Theorem}[section]

\newtheorem{lemma}[theorem]{Lemma}
\newtheorem{proposition}[theorem]{Proposition}

\theoremstyle{remark}
\newtheorem{remark}[theorem]{Remark}

\theoremstyle{remark}

\usepackage{chngcntr}
\usepackage{graphicx} 
\usepackage{float}
\counterwithout{equation}{section}
\counterwithout{theorem}{section}

\begin{document}

\title{Convex cocompact right-angled Gromov--Thurston polyhedra}
\author{Sami Douba}

\begin{abstract}
Lee and Marquis exhibited convex cocompact hyperbolic reflection groups in dimension $5$ whose limit sets are homeomorphic to the $3$-sphere, but none of whose finite-index subgroups can be realized as $4$-dimensional real hyperbolic lattices. Using different methods, we furnish right-angled examples.
\end{abstract}

\address{Mathematisches Institut der Universit\"at Bonn, Endenicher Allee 60, 53115 Bonn, Germany}
\email{douba@math.uni-bonn.de}

\maketitle

The action of the isometry group $\mathrm{Isom}(\mathbb{H}^d)$ of the real hyperbolic $d$-space $\mathbb{H}^d$ on the latter space extends to the visual boundary $\partial_\infty \mathbb{H}^d$ of $\mathbb{H}^d$, that is, to the round sphere $\mathbb{S}^{d-1}$ at infinity, providing for $d \geq 3$ an identification between $\mathrm{Isom}(\mathbb{H}^d)$ and the conformal group $\mathrm{Conf}(\mathbb{S}^{d-1})$ of $\mathbb{S}^{d-1}$. The {\em limit set} of $\Gamma$ is the $\Gamma$-invariant closed subset of $\partial_\infty\mathbb{H}^d$ consisting of all accumulation points of the $\Gamma$-orbit $\Gamma o$ for some (any) point ~$o \in \mathbb{H}^d$. One says $\Gamma$ is {\em convex cocompact} if $\Gamma$ acts cocompactly on some nonempty closed convex $\Gamma$-invariant subset of $\mathbb{H}^d$. In the latter case, the discrete subgroup $\Gamma$ is Gromov-hyperbolic as an abstract group, and the limit set of $\Gamma$ in $\partial_\infty \mathbb{H}^d$ is $\Gamma$-equivariantly homeomorphic to the Gromov boundary of $\Gamma$. In particular, if a convex cocompact subgroup $\Gamma < \mathrm{Isom}(\mathbb{H}^d)$ is (or admits a finite-index subgroup that is) abstractly isomorphic to a cocompact lattice in $\mathrm{Isom}(\mathbb{H}^{d-1})$, then the limit set of $\Gamma$ in $\partial_\infty\mathbb{H}^d$ is a topological $(d-2)$-sphere. In this note, we show the following.

\begin{theorem}\label{main}
There are convex cocompact right-angled reflection groups in $\mathrm{Isom}(\mathbb{H}^5)$ whose limit sets in $\partial_\infty\mathbb{H}^5$ are homeomorphic to $S^3$ but none of whose finite-index subgroups embed discretely in $\mathrm{Isom}(\mathbb{H}^4)$.
\end{theorem}

One then concludes from known arguments that any such reflection group will indeed fail to virtually embed as a lattice in any semisimple real algebraic group, or be quasiisometric to any symmetric space; see~\cite[\S6]{MR3968766}. We remark that it follows from Andreev's theorem~\cite{MR259734, zbMATH05176608}, together with work of Bestvina--Mess~\cite{zbMATH00020165} and Davis~\cite[Thm.~10.9.2]{zbMATH05218470}, that no reflection groups as in Theorem~\ref{main} exist if all dimensions in the statement are diminished by $1$; see~\cite[\S7]{zbMATH06147446}.

Here, a {\em reflection group $\Gamma$ in  $\mathrm{Isom}(\mathbb{H}^d)$}, or a {\em hyperbolic reflection group}, is a discrete subgroup of $\mathrm{Isom}(\mathbb{H}^d)$ generated by (hyperplane) reflections. Such $\Gamma$ is then always generated by the reflections in the walls of a locally finite convex polyhedron in $\mathbb{H}^d$ that is, up to isometry, canonically associated to $\Gamma$. If any two distinct nondisjoint walls of the latter polyhedron are in fact orthogonal, one says the polyhedron, and the reflection group~$\Gamma$, are {\em right angled}. In the context of this note, for $d \geq 3$, it will be useful to think of reflection groups in $\mathrm{Isom}(\mathbb{H}^d)$  as discrete subgroups of $\mathrm{Conf}(\mathbb{S}^{d-1})$ generated by inversions in round hyperspheres of $\mathbb{S}^{d-1} = \partial_\infty \mathbb{H}^d$.

Previously, Lee--Marquis~\cite[\S A.2]{MR3968766} exhibited reflection groups as in Theorem~\ref{main} that are not right angled (nor does it seem clear whether any of their examples admit finite-index right-angled reflection subgroups), and indeed, their argument uses crucially that their polyhedra have mutually intersecting walls. By contrast, our approach is closer to that already suggested in the work of Gromov--Thurston~\cite[\S3.7]{zbMATH04054503}. As pointed out to the author by Gye-Seon Lee, the Lee--Marquis examples are also somehow in the spirit of Gromov--Thurston since they are obtained from Esselmann's compact $4$-dimensional hyperbolic polyhedra~\cite{zbMATH00919897} by deforming certain angles, though it seems unclear whether an argument similar to that contained in this note can rule out discrete embeddings of the Lee--Marquis examples into $\mathrm{Isom}(\mathbb{H}^4)$, essentially because of the abundance of angles that are odd submultiples of $\pi$ in the latter polyhedra.

There are a few reasons that one might be interested in the right-angled aspect of the examples in Theorem~\ref{main}. For example, it is known that Theorem~\ref{main} will cease to be true if all dimensions in the statement are augmented by any fixed positive integer, since by~\cite{zbMATH00020165} and~\cite[Thm.~2]{zbMATH02021085} there is indeed no abstract Gromov-hyperbolic right-angled Coxeter group whose Gromov boundary is a sphere of dimension $\geq 4$. Moreover, convex cocompact right-angled hyperbolic reflection groups often possess a certain flexibility that has proved useful for various applications. A recent example is a construction due to Boche\'nski--Morita~\cite[Thm.~3.8]{arXiv:2501.14274} of exotic proper actions of such reflection groups on homogeneous spaces using prior work of Danciger--Gu\'eritaud--Kassel~\cite[Prop.~4.1]{zbMATH07292308}.

We also felt compelled to share this note because, though we do not pursue the following here, we expect that the aspects of the proof of our Proposition~\ref{convexcocompact} involving compact right-angled hyperbolic polyhedra (which again cease to exist above dimension $4$) can in higher dimensions be replaced by the separability result of Bergeron--Haglund--Wise~\cite{bergeron2011hyperplane} to demonstrate the existence of closed aspherical manifolds of each dimension $\geq 4$ that fail to be homotopy equivalent to any compact locally symmetric space but nevertheless admit flat conformal structures whose associated holonomy representations are convex cocompact. That such conformally flat manifolds exist was already suggested in~\cite[\S3.7]{zbMATH04054503} (see also \cite[\S1.3]{MR5111594}), and was indeed announced by M.~Kapovich~\cite[Thm.~9.9]{zbMATH05269429}, though it seems unlikely given the chronology that the immensely useful work of Bergeron--Haglund--Wise was available to Kapovich at that time.

To set up the proof of Theorem~\ref{main}, we recall that a (right-angled) hyperbolic reflection group is isomorphic to an abstract (right-angled) Coxeter group. Given a set $S$, a {\em Coxeter group} on $S$ is a group $W_S$ given by a presentation of the form
\[
W_S = \langle S \> | \> (st)^{m_{st}} = 1, \> s, t \in S \rangle 
\]
where the exponent $m_{ss} = 1$ for $s \in S$, and $m_{st} \in \{2, 3, \ldots, \infty\}$ for distinct $s, t \in S$. (An exponent $m_{st} = \infty$ indicates that no relation between $s$ and~$t$ is imposed). The datum of a Coxeter group $W_S$ is often taken to include the generating set $S$, and we will reproduce this abuse here. One says $W_S$ is {\em right angled} if $m_{st} \in \{2, \infty\}$ for all distinct $s,t \in S$, and {\em irreducible} if one cannot partition $S$ into two nonempty subsets $S_1, S_2$ such that $m_{s_1s_2}=2$ for each $(s_1,s_2) \in S_1 \times S_2$. 

We will use the following two lemmas.

\begin{lemma}\label{mostow}
Let $\Gamma$ be a finite-index subgroup of a group $W$ such that the centralizer of $\Gamma$ in $W$ is trivial, and let $d \geq 3$. If $\Gamma$ embeds as a lattice in $\mathrm{Isom}(\mathbb{H}^d)$, then so does $W$.
\end{lemma}

\begin{proof}
Suppose $\rho: \Gamma \rightarrow \mathrm{Isom}(\mathbb{H}^d)$ is such an embedding. Given $g \in W$, let $\iota_g$ be the unique element of $\mathrm{Isom}(\mathbb{H}^d)$ satisfying $g \gamma g^{-1} = \iota_g \rho(\gamma) \iota_g^{-1}$ for all $\gamma \in \Gamma$; existence of such $\iota_g$ follows from Mostow--Prasad rigidity~\cite{mostow1968quasi, prasad1973strong}, and uniqueness from the fact that a lattice in $\mathrm{Isom}(\mathbb{H}^d)$ has trivial centralizer in $\mathrm{Isom}(\mathbb{H}^d)$. The map $\overline{\rho}: W \rightarrow \mathrm{Isom}(\mathbb{H}^d)$ given by $g \mapsto \iota_g$ is then a representation extending $\rho$ whose kernel is the centralizer of $\Gamma$ in $W$, so that $\overline{\rho}$ embeds $W$ as a lattice in $\mathrm{Isom}(\mathbb{H}^d)$ under the above assumptions.
\end{proof}

\begin{lemma}\label{centralizer}
Let $S$ be finite and $W_S$  an irreducible Coxeter group on $S$ that is not virtually abelian. Then every finite-index subgroup of $W_S$ has trivial centralizer in~$W_S$.
\end{lemma}

\begin{proof}
Under the given assumptions, it follows from work of Benoist--de la Harpe~\cite{MR2081159} that, upon reducing the Tits--Vinberg representation of $W_S$, one obtains a faithful Zariski-dense representation $\rho: W_S \rightarrow \mathrm{PO}(p,q)$ for some $p, q > 0$ satisfying ${p+q \geq3}$; see~\cite[Thm.~7]{arXiv:1211.5635}. Thus, given any finite-index subgroup $\Gamma$ of $W_S$, the Zariski-closure of $\rho(\Gamma)$ in $\mathrm{PO}(p,q)$ contains $\mathrm{PSO}(p,q)$. In particular, if $g \in W_S$ centralizes $\Gamma$, then $\rho(g)$ centralizes $\mathrm{PSO}(p,q)$. Since $\mathrm{PSO}(p,q)$ has trivial centralizer in $\mathrm{PO}(p,q)$, it follows that $\rho(g) = \mathrm{Id}$, and hence $g=1$. 
\end{proof}

\begin{remark}
While we have used in the proof of Lemma~\ref{centralizer} existing literature guaranteeing the existence of convenient representations of certain Coxeter groups, we will in fact produce in the proof of Proposition~\ref{convexcocompact} Zariski-dense representations into $\mathrm{Isom}(\mathbb{H}^5) \cong \mathrm{PO}(5,1)$ of the Coxeter groups to which Lemma~\ref{centralizer} will ultimately be applied.  
\end{remark}

\begin{proof}[Proof~of~Theorem~\ref{main}]
Let $P \subset \mathbb{H}^4$ be a compact (convex) right-angled polyhedron (all currently known examples of which are commensurable to the right-angled $120$-cell), and let $\Gamma_P < \mathrm{Isom}(\mathbb{H}^4)$ be the subgroup generated by the reflections in the walls of $P$. If $\{H_s\}_{s \in S}$ is the set of walls of $P$, there is an isomorphism $W_S \rightarrow \Gamma_P$ mapping $s \in S$ to the reflection in the wall $H_s$, where $W_S$ is the abstract Coxeter group given by
\[
m_{st} = \begin{cases} 2 & \text{if the walls $H_s$ and $H_t$ meet,} \\ \infty & \text{otherwise} \end{cases}
\]
for $s, t \in S$. Since $\Gamma_P$ is a lattice in $\mathrm{Isom}(\mathbb{H}^4)$, the Coxeter group $W_S$ is easily seen to be irreducible. Fix now an adjacent pair of walls $H_{s_1}, H_{s_2} \in S$, and for any $n \geq 1$, let $W_S^n$ be the irreducible Coxeter group obtained from $W_S$ by setting $m_{s_1s_2} = 2n$ instead of $m_{s_1s_2}=2$ and keeping $m_{st}$ unchanged for all other pairs $s, t \in S$. For $s,t \in S$, we will denote by $m_{st}^{(n)}$ the order of $st \in W_S^n$. When we wish to emphasize the objects involved in the definition of $W_S^n$, we will denote the latter by $W^n(P, H_{s_1}, H_{s_2})$.

Note that for any $n \geq 1$, the nerve of $W_S^n$ in the sense of~\cite[\S7.1]{zbMATH05218470} coincides with that of $W_S$, and is thus a triangulation of $S^3$ (namely, the complex dual to the surface of $P$). It follows that the Davis complex associated to $W_S^n$, on which~$W_S^n$ acts properly and cocompactly, remains homeomorphic to $\mathbb{R}^4$ and has $\mathrm{CAT}(0)$ boundary~$S^3$; see \cite[Thm.~2.4]{MR3968766} and the references provided there. In particular, the virtual cohomological dimension of $W_S^n$ remains $4$. We first show the following (an identical argument is used in the proof of~\cite[Thm.~4.2]{bogachevdouba}, though the statement in the latter reference is not quite appropriate for our purposes).

\begin{proposition}\label{nothyperbolic}
For any $n \geq 2$, no finite-index subgroup of $W_S^n$ embeds discretely in $\mathrm{Isom}(\mathbb{H}^4)$. 
\end{proposition} 

We remark that it follows from an argument of Gromov--Thurston \cite[\S0.1]{zbMATH04054503} that there is some $m=m(P)$ for which the conclusion of Proposition~\ref{nothyperbolic} holds for all $n \geq m$. It will however be important for our purposes that $m$ be uniform in~$P$. For different purposes, Hamenst\"adt--J\"ackel came upon similar considerations in~\cite{hamenstaedtjaeckel}, and, in light of Remark~\ref{gromovthurston}, it indeed seems likely that Proposition~\ref{nothyperbolic} can be approached via \cite[Thm.~5.5]{hamenstaedtjaeckel} or \cite[Thm.~2(3)]{Hamenstaedt2026}, though our setting of reflection orbifolds allows for considerable simplifications.

\begin{proof}[Proof~of~Proposition~\ref{nothyperbolic}]
Fix $n \geq 1$, and suppose some finite-index subgroup ${\Gamma < W_S^n}$ embeds discretely in $\mathrm{Isom}(\mathbb{H}^4)$. We will show that $n=1$. Since $\Gamma$ has virtual cohomological dimension $4$, we in fact have that $\Gamma$ embeds as a cocompact lattice in $\mathrm{Isom}(\mathbb{H}^4)$. One concludes in particular that $\Gamma$ is not virtually abelian, and hence neither is $W_S^n$. By Lemmas~\ref{mostow} and~\ref{centralizer}, it follows that $W_S^n$ itself embeds as a cocompact lattice in $\mathrm{Isom}(\mathbb{H}^4)$. By~\cite[Lem.~5.4]{MR3968766}, there is then a compact convex polyhedron $P_n \subset \mathbb{H}^4$ the collection $\{H_s^n\}_{s \in S}$ of whose walls is in bijection with~$S$ and such that~$\pi/m_{st}^{(n)}$ is the dihedral angle between $H_s^n$ and $H_t^n$ for $s,t \in S$ (where $m_{st}^{(n)} = \infty$ if and only if the walls $H_s^n$ and $H_t^n$ do not meet). Set $P_1 := P$, $W_S^1:= W_S$, and $H_s^1 := H_s$ for $s \in S$. Let $T \subset S$ be the set of all $t \in S-\{s_2\}$ such that $H_t^n$ is orthogonal to $H_{s_1}^n$.

For $i\in\{1,n\}$ and $s \in S$, denote by $\sigma_s^{(i)} \in \mathrm{Isom}(\mathbb{H}^4)$ the reflection with fixed hyperplane $H_s^i$. Let $R_i \subset \mathbb{H}^4$ be the union of the images of $P_i$ under the finite dihedral group generated by $\sigma_{s_1}^{(i)}$ and $\sigma_{s_2}^{(i)}$. Then $R_i$ is a compact right-angled polyhedron, and all walls of $R_i$ that meet $H^i_{s_1}$ do so orthogonally, so that $R_i \cap H_{s_1}^i$ is a compact $3$-dimensional right-angled hyperbolic polyhedron. We now ``interbreed'' the $R_i$ as follows. By Mostow rigidity (in dimension $3$), there is an isometry $\kappa \in \mathrm{Isom}(\mathbb{H}^4)$ such that
\begin{itemize}
\item $\kappa(H_{s_1}^1) = H_{s_1}^n$;
\item $\kappa$ maps the facet of $R_1 \cap H_{s_1}^1$ with supporting hyperplane $H_t^1 \cap H_{s_1}^1$ onto the facet of $R_n \cap H_{s_1}^n$ with supporting hyperplane $H_t^n \cap H_{s_1}^n$ for each $t \in T$;
\item ${R'}:= \kappa(R_1') \cup R_n' \subset \mathbb{H}^4$ is a compact right-angled polyhedron, where $R_i'$ denotes the intersection of $R_i$ with the closed half-space of $\mathbb{H}^4$ bounded by~$H_{s_1}^i$ and containing $P_i$ for $i \in \{1,n\}$.
\end{itemize}

Let $\Gamma_{R'} < \mathrm{Isom}(\mathbb{H}^4)$ be the cocompact reflection group generated by the reflections in the walls of ${R'}$. For $s \in S$, set $\sigma_s := \sigma_s^{(n)}$ and $\tau_s := \kappa \sigma_s^{(1)} \kappa^{-1}$. Then $\sigma_s, \tau_s \in \Gamma_{R'}$ for $s \in S-\{s_1, s_2\}$, and there is an abstract automorphism $\phi : \Gamma_{R'} \rightarrow \Gamma_{R'}$ such that, for $s \in S-\{s_1, s_2\}$,
\begin{enumerate}
\item\label{order2n} $\phi(\sigma_s) = (\sigma_{s_2}\sigma_{s_1})\sigma_s(\sigma_{s_2}\sigma_{s_1})^{-1}$;
\item\label{order2} $\phi^{-1}(\tau_s) = (\tau_{s_2}\tau_{s_1})\tau_s(\tau_{s_2}\tau_{s_1})^{-1}$.
\end{enumerate}
By Mostow rigidity (in dimension $4$), there is an isometry $\iota \in \mathrm{Isom}(\mathbb{H}^4)$ such that $\phi(\gamma) = \iota \gamma \iota^{-1}$ for $\gamma \in \Gamma_{R'}$. Since the subgroup of $\Gamma_{R'}$ generated by the $\sigma_s$ (respectively, by the $\tau_s$) for $s \in S - \{s_1, s_2\}$ is Zariski-dense\footnote{More generally, as shown in~\cite[Lem.~3.2]{MR3090707}, the fundamental group of a compact hyperbolic orbifold of dimension $d \geq 3$ with (codimension-2) corners is Zariski-dense in $\mathrm{Isom}(\mathbb{H}^d)$; see also~\cite[Lem.~4.3]{bogachevdouba}.} in $\mathrm{Isom}(\mathbb{H}^4)$, we have by (\ref{order2n}) that $\iota = \sigma_{s_2}\sigma_{s_1}$ (resp., by (\ref{order2}) that $\iota= \tau_{s_1}\tau_{s_2}$). We conclude that $\sigma_{s_2}\sigma_{s_1} = \tau_{s_1}\tau_{s_2}$. But $\sigma_{s_2}\sigma_{s_1}$ has order $2n$ and $\tau_{s_1}\tau_{s_2}$ has order $2$, and so $n=1$.
\end{proof}

The following is the remaining component of the proof of Theorem~\ref{main}.

\begin{proposition}\label{convexcocompact}
For any $n \geq 1$, there exist a compact right-angled polyhedron ${P \subset \mathbb{H}^4}$ and walls $H_1, H_2$ of $P$ such that $W^n(P,H_1,H_2)$ embeds as a convex cocompact reflection group in $\mathrm{Isom}(\mathbb{H}^5)$. 
\end{proposition}

\begin{proof}[Proof~of~Proposition~\ref{convexcocompact}]
We may assume $n \geq 2$. Let $P^{120} \subset \mathbb{H}^4$ be a right-angled $120$-cell. Then $P^{120}$ induces a tiling of $\mathbb{H}^4$, as well as a tiling of each wall of $P^{120}$ by right-angled dodecahedra. Given a subset $K \subset \mathbb{H}^4$ (respectively, a subset $K \subset H$ of a wall $H$ of $P^{120}$), denote by $\mathrm{CCH}(K)$ (resp., by $\mathrm{CCH}_H(K)$) the intersection of all closed half-spaces in our tiling of $\mathbb{H}^4$ (resp., in our tiling of $H$) containing $K$. Here, ``CCH'' stands for ``coarse convex hull.'' We will use the fact that coarse convex hulls of compact sets with nonempty interior are compact right-angled polyhedra; see \cite[\S3.1]{zbMATH02081321}.

Fix two orthogonal walls $H_1, H_2$ of $P^{120}$. For $j=1,2$, let $N_j$ be the closed metric $L$-neighborhood of $H_j$ in $\mathbb{H}^4$, where $L=L(n)$ is chosen such that either component of $\partial N_1$ forms an angle of $\frac{\pi}{2n}$ with either component of $\partial N_2$; note that the existence of such $L$ is guaranteed by the intermediate value theorem. Note also that, for $j=1,2$, each wall in our tiling of $\mathbb{H}^4$ that is orthogonal to $H_j$ is also orthogonal to either component of $\partial N_j$. 

For $j=1,2$, let $Q_j \subset H_j$ be the infinite-volume $3$-dimensional polyhedron bounded by all geodesic hyperplanes in $H_j$ of the form $H \cap H_j$ where $H$ is a wall of $P^{120}$ orthogonal to both $H_1$ and $H_2$, and set $K_j := Q_j \cap N_{3-j} \subset H_j$. Then~$K_j$ is compact, and so $\mathrm{CCH}_{H_j}(K_j) \subset H_j$ is a compact right-angled $3$-dimensional polyhedron. Now view $\mathrm{CCH}_{H_j}(K_j)$ as an infinite-volume polyhedron $C_j \subset \mathbb{H}^4$ all of whose walls are orthogonal to $H_j$. Then $C_j \cap N_j$ is compact, and hence $R := \mathrm{CCH}(C_1 \cap C_2\cap(N_1 \cup N_2)) \subset \mathbb{H}^4$ is a compact right-angled polyhedron. For $j=1,2$, pick a half-space $H_j^\rightarrow$ bounded by $H_j$, and let $P:= R \cap H_1^\rightarrow \cap H_2^\rightarrow \subset \mathbb{H}^4$. Then $P$ is also a compact right-angled polyhedron, four copies of which tile $R$. Let~$H'_1$ be the component of $\partial N_1$ contained in $H_1^\rightarrow$, and $H'_2$ the component of $\partial N_2$ that is {\em not} contained in $H_2^\rightarrow$. Viewing $\mathbb{H}^4 \subset \mathbb{S}^4$ via the Poincar\'e ball model, if we replace the walls $H_1$ and $H_2$ of $P$ with $H'_1$ and $H'_2$, we obtain a $4$-dimensional polyhedron $P_n \subset \mathbb{S}^4$ bounded by round hyperspheres of $\mathbb{S}^4$ (see Figure~\ref{fig:dimensiontwo}). By the Poincar\'e polyhedron theorem~\cite[\S13.5]{zbMATH05070843} applied to the infinite-volume $5$-dimensional hyperbolic polyhedron determined by $P_n$, the subgroup $\Gamma_{P_n} < \mathrm{Conf}(\mathbb{S}^4) = \mathrm{Isom}(\mathbb{H}^5)$ generated by the inversions in the walls of $P_n$ is then naturally isomorphic to $W^n(P, H_1, H_2)$. To see that $\Gamma_{P_n}$ is convex cocompact, let $R_n \subset \mathbb{S}^4$ be the union of the images of $P_n$ under the finite dihedral group generated by the inversions in $H'_1$ and $H'_2$. Then $R_n$ is a right-angled polyhedron bounded by finitely many round hyperspheres of $\mathbb{S}^4$. Since any two walls of $P_n$ neither of which is $H'_1$ or $H'_2$ are either orthogonal or disjoint, and since each wall of $P_n$ that meets either $H'_1$ or~$H'_2$ does so orthogonally, we have that any two walls of $R_n$ are either orthogonal or disjoint. It follows that the group $\Gamma_{R_n}$ generated by the inversions in the walls of $R_n$ is convex cocompact\footnote{See, for instance, \cite[Thm.~4.7]{zbMATH06206253}; note that, while the referenced theorem states that the group generated by the reflections in the walls of a finite-sided hyperbolic Coxeter polyhedron is convex cocompact if and only if no two {\em faces} of the polyhedron (of arbitrary codimension) are asymptotic, the latter hypothesis is equivalent to the absence of asymptotic {\em walls} in the right-angled case.}, and hence so is the finite-index supergroup $\Gamma_{P_n}$ of~$\Gamma_{R_n}$. 
\end{proof}

\begin{figure}[p]
    
    \includegraphics[trim=1.35in 0.4in 1in 0in, clip, scale=0.75]{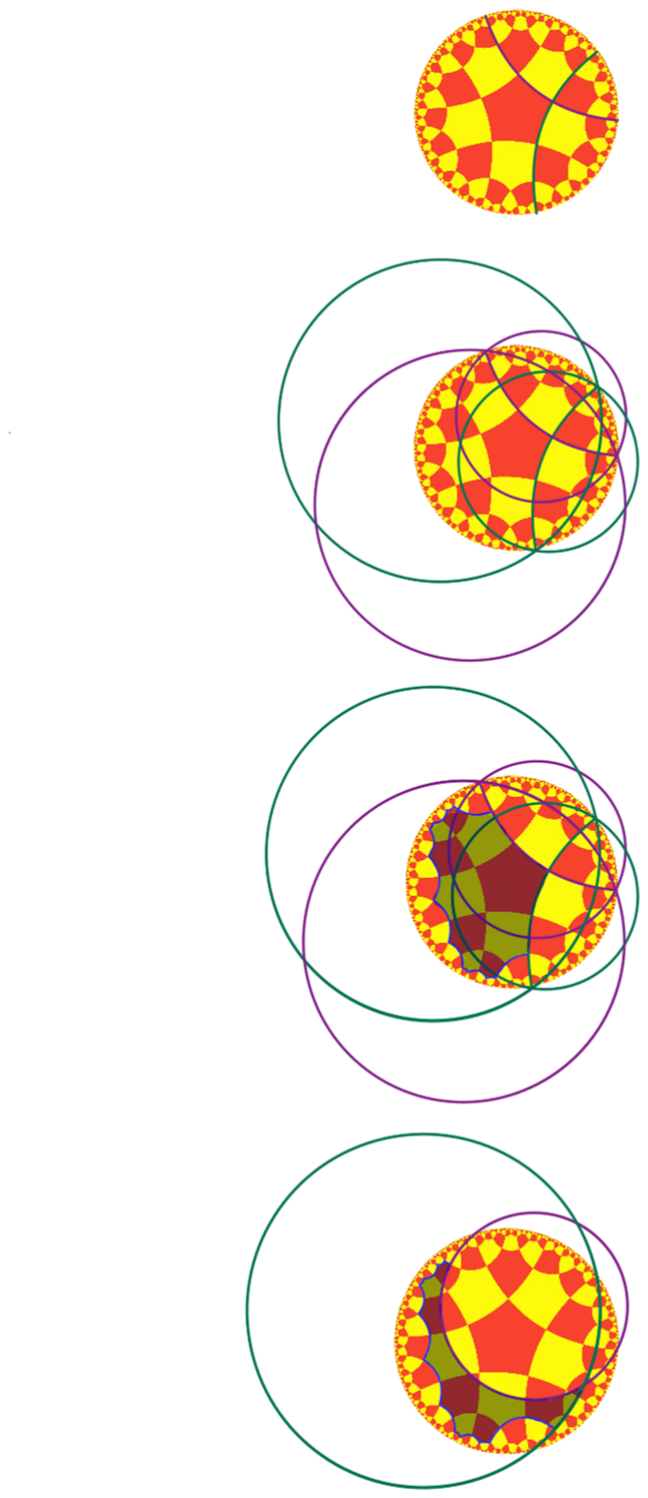} 
    \caption{A visualization, two dimensions down, of the construction of $P_n$ in the proof of Proposition~\ref{convexcocompact}.}
    \label{fig:dimensiontwo}
\end{figure}

The subgroup $\Gamma_{R_n} < \mathrm{Isom}(\mathbb{H}^5)$ from the proof of Proposition~\ref{convexcocompact} is thus a convex cocompact right-angled reflection group none of whose finite-index subgroups embed discretely in $\mathrm{Isom}(\mathbb{H}^4)$ by Proposition~\ref{nothyperbolic}, and we claim that the limit set $\Lambda \subset \partial_\infty\mathbb{H}^5$ of $\Gamma_{R_n}$ is homeomorphic to $S^3$. Indeed, by convex cocompactness of  $\Gamma_{P_n}$, we have in particular that the abstract Coxeter group $W^n(P, H_1, H_2)$ is Gromov-hyperbolic, and hence the CAT(0) boundary of the Davis complex associated to $W^n(P, H_1, H_2)$ is homeomorphic to the Gromov boundary of $W^n(P, H_1, H_2)$. The former is homeomorphic to $S^3$, and the latter is homeomorphic to the limit set in~$\partial _\infty\mathbb{H}^5$ of $\Gamma_{P_n}$ by convex cocompactness of $\Gamma_{P_n}$. But the latter limit set is nothing but $\Lambda$ since $\Gamma_{R_n}$ is a finite-index subgroup of~$\Gamma_{P_n}$.
\end{proof}

\begin{remark}\label{gromovthurston}
If we regard the polyhedra $R$ and $R_n$ in the proof of Proposition~\ref{convexcocompact} as reflection orbifolds, then $R_n$ is obtained as an $n$-fold cyclic cover of $R$ branched over $R \cap H_1 \cap H_2$ (note that $R \cap H_1 \cap H_2$ is nothing but the codimension-$2$ face of~$P$ contained in $H_1 \cap H_2$), hence the title of this note. One can then construct a negatively curved Riemannian metric on the reflection orbifold $R_n$ following~\cite[\S2]{zbMATH04054503}, and indeed, by keeping the $H'_j$ fixed in the proof of Proposition~\ref{convexcocompact} but taking larger and larger metric neighborhoods $N_j$ in the definition of $R$ (and hence in the definitions of $P$ and $R_n$), one can arrange for a negatively curved metric that is arbitrarily pinched. It also seems likely that, by following the same strategy of keeping the $H'_j$ fixed but taking larger and larger neighborhoods $N_j$, one can bring the Hausdorff dimension of the limit set $\Lambda$ arbitrarily close to $3$, though this quantity will always be strictly larger than $3$ by~\cite[Thm.~1.5]{zbMATH00880147} and Proposition~\ref{nothyperbolic}.
\end{remark}

\begin{remark}
By taking $L$ in the proof of Proposition~\ref{convexcocompact} such that the $H'_j$ instead form an angle of $\pi/3$, it is not difficult to see that the output reflection group $\Gamma_{P_n}$ and hence also the subgroup $\Gamma_{R_n}$ are contained in a cocompact arithmetic subgroup of $\mathrm{Isom}(\mathbb{H}^5)\cong \mathrm{PO}(5,1)$, namely, some conjugate $\Gamma$ of the lattice $\mathrm{PO}(f; \mathbb{Z}[\varphi])$, where~$\varphi$ denotes the golden ratio and $f$ is the quadratic form \[x_1^2 + x_2^2+x_3^2 + x_4^2 + x_5^2 -\varphi x_6^2\] on~$\mathbb{R}^6$. By keeping the $H'_j$ fixed but varying the $N_j$ as in Remark~\ref{gromovthurston}, one obtains an abundance of ``Gromov--Thurston'' reflection subgroups of the fixed lattice $\Gamma$, indeed, probably infinitely many even up to wide commensurability within $\mathrm{Isom}(\mathbb{H}^5)$. Note that, even though it was assumed in Proposition~\ref{nothyperbolic} that $m_{s_1s_2}$ was even, so that we are in principle not allowed to take the $H'_j$ forming an angle that is an {\em odd} submultiple of $\pi$, by the symmetry of the right-angled $120$-cell $P^{120}$, the proof of Proposition~\ref{nothyperbolic} still applies to show that no finite-index subgroup of $\Gamma_{P_n}$ embeds discretely in $\mathrm{Isom}(\mathbb{H}^4)$. 

The inclusion of $\Gamma_{P_n}$ in an arithmetic lattice is useful from the following perspective.   
By aforementioned work of Bergeron--Haglund--Wise~\cite{bergeron2011hyperplane}, the subgroup~$\Gamma_{P_n}$ is separable\footnote{Separability of $\Gamma_{P_n}$ in $\Gamma$ also follows from~\cite{haglund2010coxeter} in light of the fact that the lattice $\Gamma$ also happens to be reflective~\cite{zbMATH03910486}.} in $\Gamma$. By Scott's separability criterion~\cite[Lem.~1.4]{zbMATH03640531}, there is then a finite-index subgroup $\hat\Gamma < \Gamma$ containing $\Gamma_{P_n}$ such that the $1$-neighborhood of the convex core of $\Gamma_{P_n} \backslash \mathbb{H}^5$ embeds in $\hat\Gamma \backslash \mathbb{H}^5$. By cutting $\hat\Gamma \backslash \mathbb{H}^5$ along a boundary component of the convex core of $\Gamma_{P_n} \backslash \mathbb{H}^5$ and attaching the appropriate ends of $\Gamma_{P_n} \backslash \mathbb{H}^5$ to the resulting orbifold, one obtains an infinite-volume convex cocompact hyperbolic orbifold $\Delta \backslash \mathbb{H}^5$ with the property that the limit set $\Lambda_\Delta \subset \partial_\infty\mathbb{H}^5$ of $\Delta$ is a $3$-dimensional Sierpi\'nski compactum in the sense of Cannon~\cite{zbMATH03413525} each of whose peripheral $3$-spheres is the image of the limit set $\Lambda$ of $\Gamma_{P_n}$ under some conformal map of $\partial_\infty \mathbb{H}^5$. In summary, one obtains in this manner convex cocompact subgroups $\Delta < \mathrm{Isom}(\mathbb{H}^5)$ whose limit set $\Lambda_\Delta \subset \partial_\infty \mathbb{H}^5$ is a $3$-dimensional Sierpi\'nski compactum, but such that the stabilizer in $\Delta$ of each peripheral $3$-sphere of $\Lambda_\Delta$ fails to admit a discrete embedding (even up to finite index) into $\mathrm{PO}(5,1)$ stabilizing a {\em round} $3$-sphere in $\partial_\infty \mathbb{H}^5$. By contrast, any convex cocompact subgroup of $\mathrm{Isom}(\mathbb{H}^3)$ with limit set a Sierpi\'nski curve admits a convex cocompact representation into $\mathrm{Isom}(\mathbb{H}^3)$ the peripheral circles of whose limit set are all round~{\cite[\S4]{zbMATH04138789}}.

After having observed the above, we noticed that the reflection subgroup of $\mathrm{Isom}(\mathbb{H}^5)$ corresponding to the diagram $\mathcal{Q}_7$ with $p=q=7$ in~\cite[\S A.2]{MR3968766} is contained in a cocompact arithmetic subgroup of $\mathrm{Isom}(\mathbb{H}^5)$, since the Galois conjugates of the Gram matrix of $\mathcal{Q}_7$ under both nonidentity real embeddings of $\mathbb{Q}(\cos(\pi/7))$ are positive-definite. We thank Gye-Seon Lee for verifying the latter computation. The reflection group corresponding to $\mathcal{Q}_7$ is another example of a convex cocompact subgroup of $\mathrm{Isom}(\mathbb{H}^5)$ none of whose finite-index subgroups admit discrete embeddings into $\mathrm{Isom}(\mathbb{H}^4)$, and thus can alternatively be used to construct convex cocompact subgroups $\Delta < \mathrm{Isom}(\mathbb{H}^5)$ as above.
\end{remark}

\subsection*{Acknowledgements} I am grateful to Gye-Seon Lee and Franco Vargas Pallete for helpful discussions.

\bibliography{biblio}{}
\bibliographystyle{siam}

\end{document}